\documentclass[12pt,reqno]{amsart}
\usepackage{amssymb}
\usepackage{mathrsfs}
\usepackage[all]{xy}

\theoremstyle{plain}
\newtheorem{prop}{Proposition}

\newtheorem{theo}[prop]{Theorem}

\newtheorem{lemm}[prop]{Lemma}
\theoremstyle{remark}

\theoremstyle{definition}
\newtheorem{defi}[prop]{Definition}
\newtheorem{assu}[prop]{Assumption}

\newcommand{\PP}{{\mathbb P}}

\newcommand{\G}{{\mathbb G}}

\newcommand{\Z}{{\mathbb Z}}
\newcommand{\cA}{{\mathcal A}}
\newcommand{\cE}{{\mathcal E}}

\newcommand{\cH}{{\mathcal H}}
\newcommand{\cI}{{\mathcal I}}

\newcommand{\cL}{{\mathcal L}}

\newcommand{\cT}{{\mathcal T}}

\newcommand{\eqto}{\stackrel{\lower1.5pt\hbox{$\scriptstyle\sim\,$}}\to}

\DeclareMathOperator{\Pic}{Pic}
\DeclareMathOperator{\Spec}{Spec}

\DeclareMathOperator{\End}{End}
\DeclareMathOperator{\Br}{Br}

\DeclareMathOperator{\Tor}{Tor}

\begin{document}
\title[Stable rationality of surface bundles]{Stable rationality of Brauer-Severi surface bundles}
\author{Andrew Kresch}
\address{
  Institut f\"ur Mathematik,
  Universit\"at Z\"urich,
  Winterthurerstrasse 190,
  CH-8057 Z\"urich, Switzerland
}
\email{andrew.kresch@math.uzh.ch}
\author{Yuri Tschinkel}
\address{
  Courant Institute,
  251 Mercer Street,
  New York, NY 10012, USA
}

\address{Simons Foundation, 160 Fifth Av., New York, NY 10010, USA}
\email{tschinkel@cims.nyu.edu}

\date{September 28, 2017} 

\begin{abstract}
For sufficiently ample linear systems on rational surfaces we show that a very general
associated Brauer-Severi surface bundle is not stably rational.
\end{abstract}

\maketitle

\section{Introduction}
\label{sec:introduction}
This paper extends the study of stable rationality of conic bundles
over rational surfaces in \cite{HKTconic} to the case of Brauer-Severi surface bundles.
Our main result is:
\begin{theo}
\label{thm.main}
Let $k$ be an uncountable algebraically closed field of characteristic different from $3$,
$S$ a rational smooth projective surface over $k$, and
$L$ a very ample line bundle on $S$ such that $H^1(S,L)=0$, and the complete linear system
$|L|$ contains a nodal reducible curve $D=D_1\cup D_2$, where $D_1$ and $D_2$ are
smooth of positive genus, and contains a curve with $\mathsf{E}_6$-singularity.
Then the Brauer-Severi surface bundle corresponding to a very general element
of $|L|$ with nontrivial unramified cyclic degree $3$ cover is not stably rational.
\end{theo}

This is applicable, for instance, to the complete linear system
of degree $d$ curves in $\PP^2$ for $d\ge 6$.

The proof of Theorem \ref{thm.main} relies on the construction of good models of Brauer-Severi surface bundles
in \cite{KTsurf}. A new ingredient is a variant
of the standard elementary transformation of vector bundles.
This is needed to apply the specialization method, which was introduced by
Voisin \cite{voisin} and developed further
in \cite{CTP}, \cite{nicaiseshinder}, \cite{kontsevichtschinkel} and which
tells us that
in a family where one (mildly singular) member has
an obstruction to stable rationality,
the very general member fails to be stably rational.
In our case, the family is a family of Brauer-Severi surface bundles, where one member has nontrivial $3$-torsion in its unramified Brauer group.

In Section \ref{sect:basicfacts} we recall some facts on Brauer groups, and
in Section \ref{sect:elem} we describe the variant of the
standard elementary transformation that will be used in the proof
of Theorem \ref{thm.main}, which occupies Section \ref{sect:proof}.

\noindent {\bf Acknowledgments:}
We are grateful to Asher Auel for helpful comments.
The first author is partially supported by the
Swiss National Science Foundation.
The second author is partially supported by NSF grant 1601912.

\section{Basic facts}
\label{sect:basicfacts}

Recall that the Brauer group of a Noetherian scheme $S$ is defined
as the torsion subgroup of the \'etale cohomology group $H^2(S,\G_m)$ \cite{GB}.
The same definition extends to Noetherian Deligne-Mumford stacks.

In this section, we work over an algebraically closed field $k$ of characteristic different from $3$.
We start with two basic facts:

\begin{prop}[\cite{AM}]
\label{prop.Brauerexact}
Let $S$ be a smooth surface over $k$ that is \emph{(i)} projective and rational,
or \emph{(ii)} quasiprojective.
Then there are residue maps fitting in an exact sequence
\begin{align*}
0\to \Br(K)[3]\to \bigoplus_{\xi\in S^{(1)}} H^1(k(\xi)&,\Z/3\Z)
\to \bigoplus_{\xi\in S^{(2)}} \Z/3\Z&&\text{in case \emph{(i)}},\\
0\to \Br(S)[3]\to \Br(K)[3]\to&\bigoplus_{\xi\in S^{(1)}} H^1(k(\xi),\Z/3\Z)
&&\text{in case \emph{(ii)}}.
\end{align*}
Here $K=k(S)$, and $S^{(i)}$ denotes the set of codimension $i$ points of $S$.
\end{prop}

The \emph{root stack} $\sqrt[3]{(S,D)}$ along an effective Cartier divisor $D$ in $S$
is a Deligne-Mumford stack, locally, for $D$ defined by the vanishing of
a regular function $f$ on an affine chart $\Spec(A)$ of $S$, isomorphic to
the stack quotient
\[ [\Spec(A[t]/(t^3-f))/\mu_3], \]
where the roots of unity $\mu_3$ act by scalar multiplication on $t$;
cf.\ \cite[\S 2]{cadman}, \cite[App.\ B]{AGV}.
There is a closed substack with morphism to $D$ known as the
\emph{gerbe of the root stack} and given locally as
\[
[\Spec(A[t]/(t,f))/\mu_3].
\]
This is a gerbe since this $\mu_3$ acts trivially, i.e.,
\[
[\Spec(A[t]/(t,f))/\mu_3]\cong \Spec(A[t]/(t,f))\times B\mu_3,
\]
where $B\mu_3$ denotes the classifying stack of $\mu_3$.
The complement of the gerbe of the root stack maps isomorphically to $S\smallsetminus D$.

The root stack is smooth when $D$ is smooth, and singular when $D$ is singular.
For $D=D_1\cup D_2$ as in Theorem \ref{thm.main}, however, we may consider
the \emph{iterated root stack} \cite[Def.\ 2.2.4]{cadman}
\begin{equation}
\label{eqn.iteratedrootstack}
\sqrt[3]{(S,\{D_1,D_2\})}:=
\sqrt[3]{(S,D_1)}\times_S\sqrt[3]{(S,D_2)},
\end{equation}
which is smooth with stabilizer group $\mu_3$
over the smooth locus of $D$ and
$\mu_3\times \mu_3$ over $D_1\cap D_2$.
Base change by the inclusion of the gerbe of the root stack
$\sqrt[3]{(S,D_i)}$ leads to a closed substack of
$\sqrt[3]{(S,\{D_1,D_2\})}$ with morphism to the pre-image of $D_i$
in $\sqrt[3]{(S,D_{3-i})}$ which we call the
\emph{gerbe over the $i$th component}, for $i=1$, $2$:
\[ \mathfrak{D}_i\to D_i\times_S\sqrt[3]{(S,D_{3-i})}. \]

\begin{prop}[\cite{lieblicharithmeticsurface}]
\label{prop.Brauerextends}
Let $S$ be a smooth quasiprojective surface over $k$,
$D$ a curve on $S$ that is either \emph{(i)} smooth or
\emph{(ii)} nodal, consisting of two intersecting smooth components, and 
$U:=S\smallsetminus D$. Then the restriction map
induces an isomorphism
\begin{align*}
\Br\big(\sqrt[3]{(S,D)}\big)[3] \to \Br(&U)[3]
&&\text{in case \emph{(i)}},\\
\Br\big(\sqrt[3]{(S,\{D_1,D_2\})}\big)[3]\to &\Br(U)[3]
&&\text{in case \emph{(ii)}}.
\end{align*}
In each case, nonzero elements of the indicated Brauer groups are
represented by sheaves of Azumaya algebras of index $3$.
\end{prop}

In case (ii) of Proposition \ref{prop.Brauerextends}, we have a morphism
\begin{equation}
\label{eqn.rho}
\rho\colon \sqrt[3]{(S,\{D_1,D_2\})}\to \sqrt[3]{(S,D)}.
\end{equation}
Let $\alpha\in \Br\big(\sqrt[3]{(S,\{D_1,D_2\})}\big)$ be the
class of a sheaf of Azumaya algebras $\mathcal{A}$ of index $3$.

\begin{assu}
\label{assumption1}
The restriction of $\alpha$ to $\Br(U)$ does not
extend across the generic point of $D_1$ or of $D_2$ in $S$.
\end{assu}

\begin{lemm}
\label{lem.easyelementarytransformation}
With notation as above, let $x\in D_1\cap D_2$ and let
\begin{equation}
\label{eqn.projectiverep}
\mu_3\times \mu_3\to PGL_3
\end{equation}
be the projective representation associated with the restriction of
$\mathcal{A}$ to the copy of the classifying stack $B(\mu_3\times\mu_3)$
in $\sqrt[3]{(S,\{D_1,D_2\})}$ over $x$, where the factors $\mu_3$
correspond to the stabilizer along $D_1$ and along $D_2$.
Then the restriction of \eqref{eqn.projectiverep} to each factor $\mu_3$
is balanced, i.e., is isomorphic to the projectivization of the sum of the three
distinct one-dimensional linear representations of $\mu_3$.
\end{lemm}

\begin{proof}
It suffices to treat just the first factor $\mu_3$.
With the fiber product description
\eqref{eqn.iteratedrootstack} of the iterated root stack we have the
projection morphism
\[ p_2\colon \sqrt[3]{(S,\{D_1,D_2\})}\to \sqrt[3]{(S,D_2)}. \]
There is a criterion due to Alper
\cite[Thm.\ 10.3]{alpergood} for a vector bundle (e.g., the
sheaf of Azumaya algebras $\mathcal{A}$) to descend via a morphism such as $p_2$.
Specifically, Alper considers so-called good moduli spaces, e.g., the
coarse moduli space of a finite-type separated Deligne-Mumford stack over $k$
whose stabilizer groups have order not divisible by the characteristic of $k$.
However, by reasoning \'etale locally, his criterion applies as well
to a relative moduli space as in \cite[\S 3]{aovtwisted}.
Applied to $p_2$, this reveals that there exists a sheaf of Azumaya algebras
$\mathcal{A}'$ on $\sqrt[3]{(S,D_2)}$
and an isomorphism $p_2^*\mathcal{A}'\cong \mathcal{A}$
if and only if the relative stabilizer of $p_2$ acts trivially on
fibers of $\mathcal{A}$.

Now, and several times further below, we use the Kummer sequence
\[ 0\to \mu_3\to \G_m\to \G_m\to 0 \]
and the corresponding long exact sequence of cohomology groups.
We take 
\[
\alpha_0\in H^2(\sqrt[3]{(S,\{D_1,D_2\})},\mu_3)
\]
to be a lift of the
class
\[ \alpha\in \Br\big(\sqrt[3]{(S,\{D_1,D_2\})}\big)[3]. \]
To $\alpha_0$ there is a corresponding gerbe
\[ \mathfrak{G}\stackrel{\tau}\to \sqrt[3]{(S,\{D_1,D_2\})} \]
banded by $\mu_3$, meaning that $\mathfrak{G}$ is \'etale locally
over $\sqrt[3]{(S,\{D_1,D_2\})}$
isomorphic to a product with $B\mu_3$, and the
automorphism groups of the local sections are
equipped with compatible identifications with $\mu_3$.
We have $\tau^*\alpha=0$, hence
\[ \tau^*\mathcal{A}\cong \End(\cE) \]
for some rank $3$ vector bundle $\cE$ on $\mathfrak{G}$.
The stabilizer group of $\mathfrak{G}$ is a central $\mu_3$-extension $G$ of
$\mu_3\times \mu_3$:
\begin{equation}
\label{eqn.mu3extension}
1\to \mu_3\to G\to \mu_3\times \mu_3\to 1,
\end{equation}
and by convention we take $\cE$ so that the action of the
central $\mu_3$ is by scalar multplication.

The projective representation of the first factor $\mu_3$ is induced by the
linear representation of the subgroup of $G$, pre-image in \eqref{eqn.mu3extension}
of $\mu_3\times \{1\}$ in $\mu_3\times \mu_3$.
We suppose that this is not balanced.
If this is trivial then the criterion mentioned above is applicable, and
$\mathcal{A}\cong p_2^*\mathcal{A}'$ for some sheaf of
Azumaya algebras $\mathcal{A}'$ on $\sqrt[3]{(S,D_2)}$.
But then the restriction of $\alpha$ to $\Br(U)$ extends across the
generic point of $D_1$, in contradiction to our assumption.
A nontrivial unbalanced representation is the projectivization of
a linear representation which is a sum of two copies of one
and one copy of another one-dimensional linear representation of $\mu_3$.
Then the restriction of $\mathcal{E}$ to
\[ \mathfrak{G}\times_{\sqrt[3]{(S,\{D_1,D_2\})}}\mathfrak{D}_1 \]
splits canonically according to multiplicity as $\mathcal{E}_1\oplus \mathcal{E}_2$.
Let us denote by $h$ the inclusion in $\mathfrak{G}$ of the above fiber product.
Then we may form an exact sequence
\begin{equation}
\label{eqn.easyelementarytransformation}
0\to \widetilde{\mathcal{E}}^{(j)}\to \mathcal{E}\to h_*\mathcal{E}_j \to 0
\end{equation}
for $j=1$, $2$, and consider the respective corresponding
sheaf of Azumaya algebras $\widetilde{\mathcal{A}}^{(j)}$ on
$\sqrt[3]{(S,\{D_1,D_2\})}$.
Reasoning \'etale locally, we see that for appropriate $j$ the
sheaf of Azumaya algebras $\widetilde{\mathcal{A}}^{(j)}$ descends to
$\sqrt[3]{(S,D_2)}$, and we have again reached a contradiction to our assumption.
\end{proof}

\begin{assu}
\label{assumption2}
The restriction of $\alpha$ to $\Br(K)$
(where $K=k(S)$) is an element whose residue (image under the map to
$H^1(k(\xi),\Z/3\Z)$ in Proposition \ref{prop.Brauerexact})
at the generic point of $D_i$ is the class of an
\emph{unramified} cyclic degree $3$ cover $\widetilde{D}_i\to D_i$
for $i=1$, $2$.
\end{assu}

We are interested in knowing whether $\mathcal{A}$ descends to
$\sqrt[3]{(S,D)}$, i.e., is isomorphic to $\rho^*\mathcal{A}'$ for some
sheaf of Azumaya algebras $\mathcal{A}'$ on $\sqrt[3]{(S,D)}$.

\begin{lemm}
\label{lem.etalelocalkill}
With notation and assumption as above, let $x\in D_1\cap D_2$.
Then there exists an \'etale neighborhood $S'\to S$ of $x$ such that
$\alpha$ lies in the kernel of
\[ \Br(U)\to \Br(S'\times_SU). \]
\end{lemm}

\begin{proof}
We take $S'\to S$ trivializing the cyclic covers $\widetilde{D}_i\to D_i$
for $i=1$, $2$.
Application of Proposition \ref{prop.Brauerexact} to $S'$ shows that
the pullback of $\alpha$ to $\Br(S'\times_SU)$ is the restriction of an
element of $\Br(S')$.
This is trivialized upon passage to a suitable further \'etale neighborhood.
\end{proof}

\begin{prop}
\label{prop.kernel}
With notation and assumption as above, let $x\in D_1\cap D_2$.
Then the kernel of the projective representation
\eqref{eqn.projectiverep} is a subgroup, isomorphic to $\mu_3$,
embedded either as the diagonal or the antidiagonal in $\mu_3\times\mu_3$.
\end{prop}

\begin{proof}
By Lemma \ref{lem.etalelocalkill}, with its notation,
the pullback of $\alpha$ to
\[ S'\times_S\sqrt[3]{(S,\{D_1,D_2\})} \]
vanishes, and hence the projective representation lifts to a
linear representation, which is well-defined up to twist by a
character of $\mu_3\times \mu_3$ and hence may be written as
$\text{trivial}\oplus \chi\oplus \chi'$, for some characters $\chi$ and $\chi'$
of $\mu_3\times \mu_3$.
By Lemma \ref{lem.easyelementarytransformation}, the restriction of
$\chi$ and $\chi'$ to the first factor $\mu_3$ are nontrivial and opposite,
and the same holds for the restrictions to the second factor $\mu_3$.

Let $\chi_i$ for $i\in\{0,1,2\}$ denote the $i$th character of $\mu_3$.
Swapping $\chi$ and $\chi'$ if necessary, we may suppose that
\[
\chi|_{\mu_3\times\{1\}}=\chi_1\quad \text{and}\quad 
\chi'|_{\mu_3\times\{1\}}=\chi_2.
\]
Now there are two possibilities.
If 
\[
\chi|_{\{1\}\times\mu_3}=\chi_1\quad \text{and} \quad \chi'|_{\{1\}\times\mu_3}=\chi_2,
\]
then the kernel is the antidiagonal copy of $\mu_3$.
If 
\[
\chi|_{\{1\}\times\mu_3}=\chi_2\quad \text{and} \quad \chi'|_{\{1\}\times\mu_3}=\chi_1,
\]
then the kernel is the diagonal copy of $\mu_3$.
\end{proof}

\begin{defi}
\label{defn.goodbad}
In the two cases in the proof of Proposition \ref{prop.kernel}, leading to antidiagonal $\mu_3$ and
diagonal $\mu_3$, we say that the
sheaf of Azumaya algebras $\mathcal{A}$ at $x$ is \emph{good},
respectively \emph{bad}.
\end{defi}

\begin{prop}
\label{prop.goodbad}
With notation and assumption as above, 
the sheaf of Azumaya algebras $\mathcal{A}$ descends to
$\sqrt[3]{(S,D)}$ if and only if $\mathcal{A}$ is good at every point of
$D_1\cap D_2$.
\end{prop}

\begin{proof}
The morphism $\rho$ in \eqref{eqn.rho} is a relative coarse moduli space.
Indeed, if near $x\in D_1\cap D_2$ in $S$ we denote a defining equation of
$D_i$ by $f_i$ for $i=1$, $2$, then $\rho$ has the local form
\[ [\Spec(A[t_1,t_2]/(t_1^3-f_1,t_2^3-f_2))/\mu_3\times\mu_3]\to
[\Spec(A[t]/(t^3-f_1f_2))/\mu_3] \]
where $t=t_1t_2$ and $\mu_3\times \mu_3$ maps to $\mu_3$ by multiplication.
Letting $\tilde{\mu}_3$ denote the antidiagonal copy of $\mu_3$ in
$\mu_3\times\mu_3$, we obtain
\[ [\Spec(A[t_1,t_2]/(t_1^3-f_1,t_2^3-f_2))/\tilde{\mu}_3]\to
\Spec(A[t]/(t^3-f_1f_2)) \]
upon base change to an \'etale chart of $\sqrt[3]{(S,D)}$.
Triviality of the action of $\tilde{\mu}_3$ is thus necessary and
sufficient for the descent of $\mathcal{A}$ to
$\sqrt[3]{(S,D)}$.
\end{proof}

\section{Elementary transformation}
\label{sect:elem}
Already the proof of Lemma \ref{lem.easyelementarytransformation}
exhibits the use of an elementary transformation \eqref{eqn.easyelementarytransformation}
to alter the representation type of fibers of a vector bundle.
In this section we use a variant of this to change the type of a sheaf of
Azumaya algebras at a point from bad to good
(Definition \ref{defn.goodbad}).

As in the previous section, 
$S$ is a smooth quasiprojective surface over an
algebraically closed field $k$ of characteristic different from $3$, and $D=D_1\cup D_2$ is
a nodal divisor with intersecting irreducible smooth
components $D_1$ and $D_2$. We are given
nontrivial unramified cyclic degree $3$ covers
\[
\widetilde{D}_i\to D_i,\quad \text{for} \quad i=1, 2,
\]
and an element
\[
\alpha\in \Br\big(\sqrt[3]{(S,\{D_1,D_2\})}\big)[3],
\]
whose residue along $D_i$ is the class of $\widetilde{D}_i\to D_i$,
for $i=1$, $2$.
Let $\mathcal{A}$ be a sheaf of Azumaya algebras of index $3$
on $\sqrt[3]{(S,\{D_1,D_2\})}$ representing $\alpha$.
At a point $x\in D_1\cap D_2$, the sheaf of Azumaya algebras $\mathcal{A}$
has a type, good or bad, according to the type of the associated
projective representation at the point of $\sqrt[3]{(S,\{D_1,D_2\})}$ with
stabilizer $\mu_3\times \mu_3$ over $x$.

Let $C_0$ be a general nonsingular curve in $S$ through $x$.
Specifically, we suppose that $C_0$ meets $D_i$ transversely, for
$i=1$, $2$, and does not pass through any point of $D_1\cap D_2$ besides $x$.
The pre-image $C$ of $C_0$ in $\sqrt[3]{(S,\{D_1,D_2\})}$ has a
$\mathsf{D}_4$-singularity over $x$.

\begin{lemm}
\label{lem.BrCprime}
With the above notation, $\alpha$ restricts to zero in $\Br(C)$.
\end{lemm}

\begin{proof}
We argue as in \cite[Thm.\ 1.3]{harpazskorobogatov}.
Let $\widehat{C}$ denote the normalization of $C$, and $C'$ the
seminormalization:
\[ \widehat{C}\stackrel{\sigma}\to C'\stackrel{\nu}\to C. \]
Then we have an exact sequence
\[ 0\to \G_{m,C}\to \nu_*\G_{m,C'}\to i_*\cL\to 0, \]
where $\cL$ is an invertible sheaf on $B(\mu_3\times\mu_3)$,
identified with the singular substack of $C$ with inclusion map $i$.
So $\nu$ induces an isomorphism $\Br(C)[3]\to \Br(C')[3]$, and we are
reduced to showing that $\alpha$ restricts to zero in $\Br(C')$.

Identifying as well the singular substack of $C'$ with
$B(\mu_3\times\mu_3)$, with inclusion $i'$, there is an exact sequence
\[ 0\to \G_{m,C'}\to \sigma_*\G_{m,\widehat{C}}\to i'_*\cH\to 0, \]
for a two-dimensional torus $\cH$ over $B(\mu_3\times\mu_3)$, that appears
also in another exact sequence
\[ 0\to \G_{m,B(\mu_3\times\mu_3)}\to j_*\G_{m,B\tilde\mu_3}\to \cH\to 0 \]
that is related to the first by obvious restriction maps.
Here we employ the notation $\tilde\mu_3$ as in the
proof of Proposition \ref{prop.goodbad} and denote by $j$ the
morphism $B\tilde\mu_3\to B(\mu_3\times\mu_3)$.
We obtain a commutative diagram of cohomology groups
\[
\xymatrix{
\Pic(\widehat{C})\ar[r]\ar[d] & H^1(B(\mu_3\times\mu_3),\cH)\ar[r]\ar@{=}[d] & \Br(C')\ar[r]\ar[d] & 0 \\
\Z/3\Z \ar[r] & H^1(B(\mu_3\times\mu_3),\cH) \ar[r] & \Br(B(\mu_3\times\mu_3)) \ar[r] & 0
}
\]
with exact rows.
Since the map on the left is surjective,
we have an isomorphism of Brauer groups on the right.
So we are further reduced
to verifying the triviality of the restriction of $\alpha$ to
$B(\mu_3\times\mu_3)$, which holds by Lemma \ref{lem.etalelocalkill}.
\end{proof}

With the notation of the proof of Proposition \ref{prop.goodbad}
we have
\[ R_0:=A[t_1,t_2]/(t_1^3-f_1,t_2^3-f_2), \]
with $\mu_3\times\mu_3$-action, as well as twists by
characters $\chi_{i,j}$ of $\mu_3\times\mu_3$ defined by
\[
\chi_{i,j}(\lambda,\lambda'):=\lambda^i\lambda'^j.
\]
We introduce the following notation:
\begin{align*}
R_1&:=R_0\otimes \chi_{1,1},&
R_2&:=R_0\otimes \chi_{2,2},\\
R'&:=R_0\otimes \chi_{1,2},&
R''&:=R_0\otimes \chi_{2,1}.
\end{align*}
We let $I_0$ denote the ideal sheaf of $B(\mu_3\times \mu_3)$ in $C$,
with twists $I_i:=I_0\otimes \chi_{i,i}$.
Then there is an exact sequence of coherent sheaves
on a Zariski neighborhood of the point of $\sqrt[3]{(S,\{D_1,D_2\})}$
over $x$, given algebraically by
\[
0\to R_1\oplus R_2\oplus R_0\stackrel{\begin{pmatrix} -t_2^2 & t_1 & 0 \\ -t_1^2 & t_2 & 0 \\ 0 & 0 & 1 \end{pmatrix}}
{\relbar\joinrel\relbar\joinrel\relbar\joinrel\relbar\joinrel\longrightarrow}
R'\oplus R''\oplus R_0
\stackrel{\begin{pmatrix} -t_2 & t_1 & 0 \end{pmatrix}}
{\relbar\joinrel\relbar\joinrel\relbar\joinrel\relbar\joinrel\longrightarrow}
I_1\to 0.
\]
We view this as an analytic local model of an elementary transformation.

\begin{prop}
\label{prop.harderelementarytransformation}
With notation as above, we suppose that $\cA$ is bad at $x$.
Let
$\alpha_0\in H^2(\sqrt[3]{(S,\{D_1,D_2\})},\mu_3)$ be a lift of $\alpha$,
\[ \mathfrak{G}\stackrel{\tau}\to \sqrt[3]{(S,\{D_1,D_2\})} \]
a corresponding gerbe banded by $\mu_3$,
and $\cE$ a rank $3$ vector bundle on $\mathfrak{G}$ such that
$\tau^*\mathcal{A}\cong \End(\cE)$.
Then there exist a line bundle $\cL$ on $\tau^{-1}(C)$ and an exact sequence
\[ 0\to \widetilde{\cE}\to \cE\to h_*(\cI\otimes \cL)\to 0, \]
where $\cI$ denotes the ideal sheaf in $\tau^{-1}(C)$ of its singular locus,
as a reduced substack, and $h$ denotes the inclusion $\tau^{-1}(C)\to \mathfrak{G}$.
Furthermore, the sheaf $\widetilde{\cE}$ on the left is locally free and determines
a sheaf of Azumaya algebras $\widetilde{\mathcal{A}}$ on
$\sqrt[3]{(S,\{D_1,D_2\})}$ that is good at $x$.
\end{prop}

\begin{proof}
Lemma \ref{lem.BrCprime} tells us that there is a line bundle
$\cT$ on
\[ \mathfrak{G}\times_{\sqrt[3]{(S,\{D_1,D_2\})}}C \]
for which the induced character of the constant $\mu_3$ stabilizer is $\chi_1$.
Consequently, the restriction of $\cE$, tensored with $\cT^\vee$,
descends to a vector bundle $E$ on $C$.
Since we are free to twist $\cT$ by the pullback of any line bundle from $C$,
there is no loss of generality in supposing that the isomorphism type of
$E$ over $x$ is $\chi_{1,2}\oplus \chi_{2,1}\oplus \chi_{0,0}$.

Let $L$ be a line bundle on $C$ whose isomorphism type over $x$ is
$\chi_{1,1}$.
We let $I$ denote the ideal sheaf in $C$ of its singular locus
(as a reduced substack); the fiber of $I$
at the point over $x$ is a two-dimensional vector space
with representation $\chi_{1,0}\oplus \chi_{0,1}$.
So there exists an equivariant surjective linear map from the
fiber of $E$ to the fiber of $I\otimes L$.
This extends to a morphism of modules
(first non-equivariantly, then equivariantly by averaging),
which we may view as a surjective morphism of sheaves
\[ E|_{V\times_SC}\to (I\otimes L)|_{V\times_SC}, \]
for some affine neighborhood $V\subset S$ of $x$.
As explained in \cite[\S 4.3]{KTsurf} this extends,
after possibly modifying $L$ away from $x$, to a surjective morphism
of sheaves on $C$.
Pulling back to the gerbe and tensoring with $\cT$ determines a
surjective morphism of sheaves on $\mathfrak{G}$ and hence an exact sequence
as in the statement.

The ideal sheaf $\cI$ is Cohen-Macaulay of depth $1$, so
by the Auslander-Buchsbaum formula has projective dimension $1$,
and $\widetilde{\cE}$ is locally free.

For the analysis of the type of the sheaf of Azumaya algebras
$\widetilde{\mathcal{A}}$ at $x$, which is sensitive only to the
projective representation of the
$\mu_3\times \mu_3$ stabilizer over $x$,
we may pass to an \'etale neighborhood of $x\in S$ as in
Lemma \ref{lem.etalelocalkill} and thus assume that we have
an exact sequence as in the statement of the proposition on
$\sqrt[3]{(S,\{D_1,D_2\})}$, rather than on a gerbe.
As before, $\cE$ is only determined up to twisting by a line bundle.
Since the map from the Picard group of $\sqrt[3]{(S,\{D_1,D_2\})}$ to the
character group of $\mu_3\times \mu_3$ (given by restriction to the
copy of $B(\mu_3\times \mu_3)$ over $x$) is surjective,
there is no loss of generality in supposing as before that the
isomorphism type of
$\cE$ over $x$ is $\chi_{1,2}\oplus \chi_{2,1}\oplus \chi_{0,0}$,
and of the coherent sheaf on the right is
$\chi_{1,2}\oplus \chi_{2,1}$.
Restriction to the copy of $B(\mu_3\times \mu_3)$ over $x$ determines a
four-term exact sequence with a $\Tor$ sheaf on the left
\[ 0\to \Tor\to \widetilde{\cE}|_{B(\mu_3\times \mu_3)}\to
\chi_{1,2}\oplus \chi_{2,1}\oplus \chi_{0,0}\to
\chi_{1,2}\oplus \chi_{2,1}\to 0. \]
Since the configuration of $D_1$, $D_2$, and $C$ in $S$ at $x$ has a
unique analytic isomorphism type, the model computation just before
the statement of the proposition may be used to see that
\[ \Tor\cong \chi_{1,1}\oplus \chi_{2,2}. \]
It follows that $\widetilde{\mathcal{A}}$ is good at $x$.
\end{proof}

\section{Proof of the main theorem}
\label{sect:proof}
The argument begins as in the proof of the main
theorem of \cite{HKTconic}.
The hypotheses guarantee that the monodromy action on
nontrivial unramified cyclic degree $3$ covers of a nonsingular member
of $|L|$ is transitive; cf.\ the proof of \cite[Lem.\ 3.1]{HKTmoduli}.
We take the space of reduced nodal curves in $|L|$ with nontrivial degree $3$
cyclic \'etale covering, and the member $D=D_1\cup D_2$ with
degree $3$ cyclic \'etale cover, nontrivial over each component,
as pointed variety $(B,b_0)$.
There is an associated element
\[ \alpha\in \Br\big(\sqrt[3]{(S,\{D_1,D_2\})}\big), \]
by Propositions \ref{prop.Brauerexact} and \ref{prop.Brauerextends},
represented by a sheaf of Azumaya algebras $\mathcal{A}$ of index $3$.
By repeated application of Proposition \ref{prop.harderelementarytransformation},
we may suppose that $\mathcal{A}$ is good at all nodes of $D$.
By Proposition \ref{prop.goodbad}, $\mathcal{A}$ descends to the
(singular) root stack $\sqrt[3]{(S,D)}$; we let
\[ \beta\in \Br\big(\sqrt[3]{(S,D)}\big) \]
denote its Brauer class, and
\[ \gamma\in H^2\big(\sqrt[3]{(S,D)},\mu_3\big) \]
a choice of lift, with gerbe $\mathfrak{G}_0$ associated with $\gamma$
and locally free coherent sheaf $\cE_0$ of rank $3$ associated with
the sheaf of Azumaya algebras.

Applying the deformation-theoretic machinery of
\cite[\S 4.3]{HKTconic}, we obtain by (usual) elementary transformation
a subsheaf $\widetilde{\cE}_0$, also locally free of rank $3$, for
which the space of obstructions vanishes.
Upon replacing $B$ by a suitable \'etale neighborhood of $b_0$,
we obtain the root stack $\sqrt[3]{(B\times S,\mathcal{D})}$,
where $\mathcal{D}$ denotes the corresponding family of divisors
in $B\times S$, class
\[ \Gamma\in H^2\big(\sqrt[3]{(B\times S,\mathcal{D})},\mu_3\big) \]
restricting to $\gamma$, gerbe
\[ \mathfrak{G}\to \sqrt[3]{(B\times S,\mathcal{D})} \]
restricting to $\mathfrak{G}_0$, and locally free sheaf
$\widetilde{\cE}$ on $\mathfrak{G}$ restricting to $\widetilde{\cE}_0$.
The locally free sheaf $\widetilde{\cE}$ determines a smooth $\PP^2$-bundle
\[
\mathcal{P}\to \sqrt[3]{(B\times S,\mathcal{D})}.
\]

We now apply the final step in the proof of \cite[Thm.\ 1.4]{KTsurf}
to the $\PP^2$-bundle $\mathcal{P}$.
The construction of good models of Brauer-Severi surface
bundles from op.\ cit., applied to $\mathcal{P}$ produces a
Brauer-Severi surface bundle
\[ \mathcal{X}\to B\times S. \]
Over $B$, this is a flat family of Brauer-Severi surface bundles over $S$.
Since the discriminant curve of the fiber over $b_0$ has two components, and
the Brauer class is given by nontrivial \'etale cyclic covers,
this fiber has nontrivial unramified Brauer group \cite{AM}.
Such a Brauer-Severi surface bundle has singularities of toric type, and
these are mild enough for the specialization method to be applicable.
We conclude that the very general Brauer-Severi surface bundle in
this family is not stably rational.

\bibliographystyle{plain}
\bibliography{pn}

\end{document}